\documentclass[12pt]{amsart}
\usepackage[american]{babel}
\usepackage[utf8x]{inputenc} 
\usepackage{amsfonts, latexsym, amssymb, amsgen, amsbsy, amstext, amsopn, amsmath, txfonts, mathrsfs,yfonts,dsfont}
\usepackage{cmll}
\usepackage{amsthm}
\usepackage{color}
\usepackage[all]{xy}
\usepackage{fancyhdr}
\usepackage{newlfont}
\usepackage{setspace}
\usepackage{verbatim}
\usepackage{hyperref} 

\AtBeginDocument{\def\MR#1{}} 

\usepackage[T1]{fontenc}
\usepackage{amsthm}
\usepackage{color}

\newcommand{\G}{\mathbb{G}}

\newcommand{\R}{\mathbb{R}}

\newcommand{\vH}{\varmathbb{H}}

\newcommand{\lan}{\langle}
\newcommand{\ran}{\rangle}

\newcommand{\pa}[1]{\left( #1 \right)}               
\newcommand{\set}[1]{\left\{ #1 \right\}}            
\newcommand{\ban}[1]{\left\langle  #1 \right\rangle}  
\newcommand{\bann}{\left\langle \cdot,\cdot \right\rangle} 

\theoremstyle{plain}
\newtheorem{theorem}{Theorem}[section]

\newtheorem{proposition}[theorem]{Proposition}

\newtheorem{remark}[theorem]{Remark}

\newcommand{\norm}[1]{\left\lVert#1\right\rVert}
\DeclareMathOperator{\Lie}{Lie}
\DeclareMathOperator{\ad}{ad}
\DeclareMathOperator{\spn}{span}

\begin{document}
	
\title[Michor--Mumford phenomenon in the infinite dimensional Heisenberg group]
{On the Michor--Mumford phenomenon in the infinite dimensional Heisenberg group}
\author{Valentino Magnani}
\address{Valentino Magnani, Dip.to di Matematica, Universit\`a di Pisa \\
	Largo Bruno Pontecorvo 5 \\ I-56127, Pisa}
\email{valentino.magnani@unipi.it}
\author{Daniele Tiberio}
\address{Daniele Tiberio, 
	SISSA, via Bonomea 265, 34136 Trieste, Italy}
\email{dtiberio@sissa.it}


\begin{abstract}
In the infinite dimensional Heisenberg group, we construct a left invariant weak Riemannian metric that gives a degenerate geodesic distance. The same construction
yields a degenerate sub-Riemannian distance.
We show how the standard notion of sectional curvature 
adapts
to our framework, but it cannot be defined everywhere and it is unbounded on suitable sequences of planes.
The vanishing of the distance precisely occurs along this sequence of 
planes, so that the degenerate Riemannian distance appears in connection
with an unbounded sectional curvature.
In the 2005 paper by Michor and Mumford,
this phenomenon was first observed in some specific Fréchet manifolds.

\end{abstract}

\maketitle

\tableofcontents

\section{Introduction}

Geodesic distances naturally appear in the geometry of infinite dimensional manifolds.
A new aspect is that they may also vanish on two distinct points.
In general, the vanishing of the geodesic distance may occur for certain Riemannian metrics,
where no special conditions are assumed, namely for
{\em weak Riemannian metrics}, \cite[Definition~[5.2.12]{AMR88}.
These metrics are important, since they are the only possible metrics
when the manifold is not modelled on a Hilbert space. 

First examples of vanishing geodesic distances in infinite dimensional 
Fréchet manifolds 
were found in \cite{EP1993}, \cite{MM05} and \cite{MM06}.
A simple example of vanishing geodesic distance can be also constructed 
in a Hilbert manifold, \cite{MT20}. 
However, one may still wonder whether replacing a weak Riemannian metric
with a left invariant weak Riemannian metric with respect to a Hilbert
Lie group structure somehow might give a condition
to have positive geodesic distance on distinct points.

The answer to this question does not seem intuitively clear.
For instance, we observe that connected, simply connected
and commutative Banach Lie groups, equipped with a bi-invariant weak Riemannian metric
have positive geodesic distance on distinct points. 
In short, their geodesic distance is actually a distance. 
The proof of this fact essentially follows from \cite[Proposition~IV.2.7]{Nee06:Tow},
observing that the exponential mapping is a local Riemannian isometry.

Thus, the question is whether considering a left invariant weak Riemannian metric on a noncommutative, connected and simply connected Banach Lie group may prevent the vanishing of the geodesic distance. Our first result answers this question in the negative. 
\begin{theorem} \label{thm:degendistHeis}
There exists a left invariant weak Riemannian metric on the infinite 
dimensional Heisenberg group $\vH$, whose associated geodesic distance is not positive on all couples of distinct points.
\end{theorem}
The Heisenberg group $\vH$ is modelled on the Hilbert space $\ell^2\times\ell^2\times\R$,
where $\ell^2$ is the standard linear space of square-summable sequences.
More details are given in Section~\ref{sect:preliminaries}.
The same technique to prove the previous theorem also 
gives an analogous degenerate geodesic distance for the sub-Riemannian Heisenberg group.
\begin{theorem} \label{thm:degenSRdistHeis}
	There exists a left invariant weak sub-Riemannian metric on the infinite 
	dimensional Heisenberg group $\vH$ such that its associated geodesic distance is not positive on all couples of distinct points.
\end{theorem}
The previous theorems are contained in Theorem~\ref{Thm}
and their proof relies on the same sequence of length-minimizing curves.
Furthermore, the proof of these results precisely shows that both Riemannian and sub-Riemannian distance are vanishing between points that have the same projection
on the subspace $\ell^2\times\ell^2\times\{0\}$.
Remark~\ref{rem:positive} completes the picture, showing that 
when the projections of two points on $\ell^2\times\ell^2\times\{0\}$
are different, then both their Riemannian and sub-Riemannian distance are positive.

From another perspective, dealing with a left invariant weak Riemannian metric has the advantage to find the sectional
curvature by more manageable formulas.
In \cite{MM05}, Michor and Mumford proved that in different Fréchet manifolds with a vanishing geodesic distance the sectional curvature is unbounded.
Theorem~\ref{thm:unboundedSectCurv} below presents the same phenomenon for the left invariant weak Riemannian metric $\sigma$ 
defined in \eqref{def:sigma0}, in the infinite
dimensional Heisenberg group $\vH$.

We wish to emphasize that for general weak Riemannian metrics the existence of the Levi--Civita connection is not guaranteed a priori, hence the same existence problem 
involves the sectional curvature.
From the standard formula for the sectional curvature of Lie groups, 
see for instance \cite{Arn66} and \cite{BBM14}, we notice that 
the sectional curvature  of $\vH$ with respect to $\sigma$
is well defined on ``many planes'' of the Lie algebra $\Lie(\vH)$.
We also observe that the ``finite dimensional formula'' for the sectional curvature through the structure coefficients of $\Lie(\vH)$, \cite[Lemma~1.1]{Milnor76}, converges on the previous planes
to the same sectional curvature obtained by \cite[Theorem~5]{Arn66}.
Broadly speaking, we may think that the convergence of the sectional curvature
in Milnor's paper \cite{Milnor76} could be interpreted as a computation of sectional curvature of $\vH$ through a finite dimensional approximation 
by an orthonormal basis. 
On the other side, we also observe that this convergence
does not hold on all 2-dimensional subspaces of 
$\Lie(\vH)$, as shown in Remark~\ref{rem:undefinedSectCurv}.
In addition, according to Proposition~\ref{prop:discont}, we can also prove that 
our sectional curvature is discontinuous exactly at the
plane where it cannot be defined.

\begin{theorem}\label{thm:unboundedSectCurv}
Let $\vH$ be the infinite dimensional Heisenberg group equipped with the
left invariant weak Riemannian metric $\sigma$.
Then there exists two sequences of orthonormal vectors $a_{1j},a_{2j}\in \Lie(\vH)$
and $b\in \Lie(\vH)$ with $j\ge1$ such that
$K_\sigma(a_{1j},b)=K_\sigma(a_{2j},b)$, 
\[
\lim_{j\to\infty}K_\sigma(a_{1j},a_{2j})=-\infty\quad\text{and}\quad \lim_{j\to\infty}K_\sigma(a_{1j},b)=+\infty.
\]
The numbers $K_\sigma(a_{1j},a_{2j})$ and $K_\sigma(a_{1j},b)$
are the sectional curvatures of the planes of $\Lie(\vH)$
spanned by the orthonormal bases $(a_{1j},a_{2j})$ 
and $(a_{1j},b)$.
\end{theorem}
The proof of this theorem is provided in Section~\ref{sect:remSectCurv},
where also more information on the vectors $a_{1j},a_{2j}$
and $b$ can be found. Inspecting the proofs of Theorem~\ref{Thm} and 
Theorem~\ref{thm:unboundedSectCurv} another interesting phenomenon appears.
The curves whose lengths converge to zero and that connect two distinct points 
are precisely contained in the {\em span of the planes} where the
sectional curvature blows-up.

\section{Preliminary notions}\label{sect:preliminaries}

We denote by $\ell^2$ the linear space of all real and square summable sequences. 
Its scalar product $\lan \cdot,\cdot\ran$ has associated norm 
$\|x\|=\sqrt{\sum_{j=1}^\infty x_j^2}$ for any element $x=\sum_{j=1}^\infty x_je_j$.
The set of unit vectors $\{e_j:j\ge1\}$ denotes the canonical orthonormal basis of $\ell^2$.
For each integer $n\ge1$, the element $e_n$ of $\ell^2$ has $n$-th entry equal to 
$1$ and all the others are zero.

We consider $\ell^2 \times \ell^2\times \mathbb{R}$ endowed with its standard structure of product of Hilbert spaces.
We also equip this space with a noncommutative 
Lie group operation 
\begin{equation}
(h_1,h_2,\tau)(h_1',h_2',\tau')=(h_1 + h_1', h_2 + h_2', \tau + \tau' + \beta( (h_1,h_2), (h_1',h_2') ) )
\end{equation}
for all elements $(h_1,h_2,\tau),(h_1',h_2',\tau')\in\ell^2\times\ell^2\times\R$, 
where $ \beta: (\ell^2 \times \ell^2) \times (\ell^2 \times \ell^2)\to \mathbb{R}$ is given by 
\begin{equation}\label{eq:betaBCH}
\beta( (h_1,h_2), (h_1',h_2') ) = \langle h_1 , h_2'\rangle - \langle h_2 , h_1' \rangle.
\end{equation}
We denote by $\vH$ the Hilbert Lie group arising from the previous group operation, that is the 
{\em infinite dimensional Heisenberg group} modelled on 
the Hilbert space $\ell^2\times\ell^2\times\R$.
The previous Lie group operation yields the Lie product
\begin{equation}\label{eq:LieProdH}
[(h_1,h_2,\tau),(h_1',h_2',\tau')]=2\,\beta( (h_1,h_2), (h_1',h_2') )\, (0,0,1),
\end{equation}
that makes $\vH$ also an {\em infinite dimensional Heisenberg Lie algebra}.
For each $p\in\vH$, we denote by $L_p:\vH\to\vH$ the left multiplication by $p$, 
defined as $L_p(r)=p\cdot r $ for all $r\in\vH$. 
The group operation gives the following simple formula for the differential of $L_p$ at a point $q$, namely
$$
(dL_p)_q(v) = \lim_{t \to 0} \frac{L_p(q+tv) - L_p(q)}{t} = (v_1,v_2,v_3 + \langle p_1,v_2 \rangle - \langle p_2 , v_1\rangle )
$$
for every $v=(v_1,v_2,v_3) \in T_q\vH$, with $p=(p_1,p_2,p_3)$. 

We have used a canonical identification between $T_q\vH$ and $\vH$, being $\vH$ 
a Hilbert manifold equipped with a structure of topological vector space.
We also notice that our formula for the differential $(dL_p)_q$ does not depend on the point $q$. 

\subsection{Left invariant weak Riemannian metrics}
We consider 
a continuous scalar product 
$$\sigma_0:T_0\vH\times T_0\vH\to\R$$ 
on the tangent space $T_0\vH$ of $\vH$ at the origin.
Then for every $p \in\vH$ and $v,w \in T_p\vH$ 
the following scalar product
\begin{equation} \label{leftinvmetric:defr}
	\sigma_p(v,w) = \sigma_0 \big((dL_{p^{-1}})_pv , (dL_{p^{-1}})_pw \big)
	=\sigma_0 \big(({dL_{-p}})_pv , ({dL_{-p}})_pw \big)
\end{equation} 
defines a {\em left invariant weak Riemannian metric $\sigma$ on $\vH$}. If for any piecewise smooth curve
$\gamma:[0,1]\to\vH$ we define its Riemannian length as
\[
\ell_\sigma(\gamma) = \int_{0}^{1} \norm{ \dot{\gamma} (t) }_\sigma dt,
\]
then the associated geodesic distance $d: \vH\times\vH\to [0,+ \infty)$ between $p_1,p_2\in\vH$ is
\begin{equation}\label{def:riem}
	d(p_1,p_2)=\inf\{ \ell_\sigma(\gamma): \gamma \text{ is a piecewise smooth curve with } \gamma(0)=p_1, \gamma(1)=p_2\}.
\end{equation}
It is plain to check that $d$ is left invariant, is symmetric 
and satisfies the triangle inequality.

Taking into account the canonical identification between $\vH$ and $T_0\vH$, 
the set $\ell^2 \times \ell^2 \times \{0\}$ can be seen as a closed subspace of
$T_0\vH$, that we denote by $H_0\vH$.
Then we obtain a left invariant \textit{horizontal subbundle}, denoted by $H\vH$, whose fibers are
$$H_p\vH = (dL_p)_0 (H_0\vH)\subset T_p\vH$$
for every $p=(p_1,p_2,p_3)\in\vH$.
We note that $v=(v_1,v_2,v_3) \in H_p\vH$ if and only if 
\begin{equation}\label{eq:dLp}
(dL_{-p})_p(v)  =(v_1,v_2, v_3 - \langle p_1,v_2 \rangle + \langle p_2 , v_1\rangle) \in H_0\vH
\end{equation}
and the previous condition corresponds to the equality 
$$
v_3 - \langle p_1,v_2 \rangle + \langle p_2 , v_1\rangle =0.
$$
We have a precise formula to define the {\em horizontal curves}
associated to $H\vH$. They are continuous and piecewise smooth curves 
$\gamma:[0,1] \to\vH$
of the form $\gamma=(\gamma_1,\gamma_2,\gamma_3)\in\vH$, such that 
for almost every $t \in [0,1]$ we have 
$$
\dot{\gamma}_3(t) - \langle \gamma_1(t),\dot{\gamma}_2(t) \rangle + \langle \gamma_2(t) , \dot{\gamma}_1(t)\rangle = 0.
$$
The previous differential constraint means that $\dot{\gamma}(t) \in H_{\gamma(t)}\vH$.

On the horizontal fibers $H_p\vH$ of $H\vH$ we can fix a scalar product. 
A \textit{left invariant weak sub-Riemannian metric} $g$ on $H\vH$ is defined
by a continuous inner product 
\[
g_0:H_0\vH\times H_0\vH\to\R,
\]
such that for all $p \in\vH$ and $v,w \in H_p\vH$ we have 
\begin{equation} \label{leftinvmetric:def}
g_p(v,w) = g_0 \big((dL_{p^{-1}})_pv , (dL_{p^{-1}})_pw \big)  = g_0 \big(({dL_{-p}})_pv , ({dL_{-p}})_pw \big).
\end{equation} 
The associated {\em weak sub-Riemannian norm} is denoted by $\|\cdot\|_g$
and the length of a horizontal curve $\gamma:[0,1] \to \vH$ is defined by 
$$\ell_g(\gamma) = \int_{0}^{1} \norm{ \dot{\gamma} (t) }_g dt. $$
For any couple of points in $\vH$,
it is easy to construct a piecewise smooth horizontal curve that connects them,
hence the following \textit{sub-Riemannian distance} 
\begin{equation}\label{defrho}
\rho(p_1,p_2)=\inf\{ \ell_g(\gamma): \gamma \text{ is a horizontal curve with } \gamma(0)=p_1, \gamma(1)=p_2\}
\end{equation}
is finite for every couple of points $p_1,p_2\in\vH$,
hence we have $\rho: \vH\times\vH\to [0,+ \infty)$.
One may easily observe that $\rho$ is left invariant, symmetric
and satisfies the triangle inequality.

\section{Degenerate geodesic distances in the infinite dimensional Heisenberg group}\label{sect:GeodDist}

This section is devoted to the construction of special left invariant weak Riemannian (and sub-Riemannian) metrics that yield degenerate geodesic distances. 

We introduce the linear and continuous operator $A:\ell^2 \to \ell^2$, which associates to each $x\in\ell^2$ of components $(x_k)_{k\ge1}$ the element $Ax \in \ell^2$, whose $k$-th component is 
$(Ax)_k= x_k/k$.
Then we define the scalar product $\eta: \ell^2 \times \ell^2 \to \mathbb{R}$ as
$$\eta(v,w)= \langle Av,w \rangle$$ 
for all $v,w\in\ell^2$.
We use $\eta$ to define the new scalar product 
\begin{equation}\label{def:g_0}
	g_0((v_1,v_2),(w_1,w_2)) = \eta(v_1,w_1) + \eta(v_2,w_2)
\end{equation}
for every $(v_1,v_2),(w_1,w_2)\in\ell^2\times\ell^2$.
By our identification, $g_0$ 
can be seen as a scalar product on $H_0\vH$, so that using (\ref{leftinvmetric:def}) we obtain a left invariant weak sub-Riemannian metric 
$g$ on $\vH$.
We follow the notation of the previous section,
denoting by $\rho$ the special sub-Riemannian distance associated to 
this choice of $g$ through formula \eqref{defrho}.

To obtain a left invariant weak Riemannian metric $\sigma$ on $\vH$, we extend $g_0$ as follows
\begin{equation}\label{def:sigma0}
\sigma_0((v_1,v_2,v_3),(w_1,w_2,w_3)=g_0((v_1,v_2),(w_1,w_2))+v_3w_3
\end{equation}
for every $(v_1,v_2,v_3),(w_1,w_2,w_3)\in T_0\vH$, where $\sigma_0:T_0\vH\times T_0\vH\to\R$.
From \eqref{leftinvmetric:defr}, the scalar product in \eqref{def:sigma0}
immediately defines a left invariant weak Riemannian metric $\sigma$ on $\vH$.
The Riemannian distance associated to $\sigma$ through \eqref{def:riem}
will be denoted by $d$.
From the definitions of $g$, $d$ and $\rho$, one immediately observes that
$d\le \rho$.

\begin{remark}\label{rem:positive}\em 
It is easy to notice that both $d$ and $\rho$ are not everywhere vanishing on $\vH$.
We consider $(p_1,p_2,t), (q_1,q_2,s)\in \vH$ with $(p_1,p_2) \neq (q_1,q_2)$ and
we choose any piecewise smooth curve $\gamma=(\gamma_1,\gamma_2,\gamma_3):[0,1] \to \vH$
with $\gamma(0)=(p_1,p_2,t)$ and $\gamma(1)=(q_1,q_2,s)$. 
Let $i_0\in\{1,2\}$ be such that $p_{i_0}\neq q_{i_0}$ and let $k_0\ge1$ such that 
$p_{i_0k_0}\neq q_{i_0k_0}$, where 
$$p_{i_0}=\sum_{j=1}^\infty p_{i_0j}e_j\quad \text{and}\quad q_{i_0}=\sum_{j=1}^\infty q_{i_0j}e_j. $$
We consider the component $\gamma_{i_0}=\sum_{j=1}^\infty \gamma_{i_0j}e_j$ and the following inequalities 
\begin{equation*}
	\begin{split}
		\ell_\sigma(\gamma) 
		& \ge \int_{0}^{1} \sqrt{ \norm{\dot{\gamma}_1}_\eta^2 + \norm{\dot{\gamma}_2}_\eta^2 } dt \\
		& \geq \int_{0}^{1} \norm{\dot{\gamma}_{i_0}}_\eta dt \geq \int_{0}^{1} \frac{|\dot{\gamma}_{i_0k_0}|}{\sqrt k_0} dt\ge \frac{|p_{i_0k_0} - q_{i_0k_0}|}{\sqrt{k_0}}>0.
	\end{split}
\end{equation*}
In particular, we have shown that
\[
0<\frac{|p_{i_0k_0} - q_{i_0k_0}|}{\sqrt{k_0}}\le d((p_1,p_2,t),(q_1,q_2,s))
\le  \rho((p_1,p_2,t),(q_1,q_2,s)).
\]
The previous computation also shows that both $d$ and $\rho$ are actually distances, if restricted to any hyperplane $\ell^2 \times \ell^2 \times \{\kappa\}$ with $\kappa\in\R$.
\end{remark}

We are now in a position to prove the following theorem.

\begin{theorem} \label{Thm}
	There exist a left invariant weak sub-Riemannian metric and a left invariant weak Riemannian metric on $\vH$ such that their associated geodesic distances are not positive on all couples of distinct points.
\end{theorem}
\begin{proof}
For each $p \in\vH$, we denote the norm of a horizontal vector 
$$v=(v_1,v_2,v_3) \in H_p\vH$$ with respect to $g$ as follows
\begin{equation}\label{eq:vg}
\norm{v}_{g} = \norm{(dL_{-p})_pv}_{g} = \norm{(v_1,v_2,0)}_{g},
\end{equation}
where the last equality is due to \eqref{eq:dLp}
and $(v_1,v_2,0)$ is identified with a vector of $H_0\vH$.

Since the subspaces $\ell^2 \times \{0\} \times \{0\}$ and $\{0\} \times \ell^2 \times \{0\}$ of $H_0\vH$ are orthogonal with respect to $g_0$, the
previous equalities give
$$
\norm{v}^2_{g} =  \norm{v_1}_\eta^2 + \norm{v_2}_\eta^2, 
$$
where we have defined the norm
\begin{equation}\label{def:normB}
\|u\|_\eta=\sqrt{\eta(u,u)}=\sqrt{\langle Au,u\rangle}
\end{equation}
for every $u\in\ell^2$. 
As a consequence, the length of a horizontal curve 
$\gamma:[0,1] \to \vH$ with respect to $g$ 
satisfies the formula
\begin{equation}\label{def:lg}
\ell_g(\gamma) =\int_{0}^{1} \sqrt{ \norm{\dot{\gamma}_1}_\eta^2 + \norm{\dot{\gamma}_2}_\eta^2 } dt,
\end{equation}
where $\gamma(t)= (\gamma_1(t), \gamma_2(t), \gamma_3(t))$.

Next, we wish to show that whenever
$(p_1,p_2,s_1), (p_1,p_2,s_2)\in\vH$, then
\begin{equation}\label{eq:dg0}
\rho((p_1,p_2,s_1), (p_1,p_2,s_2))=0.
\end{equation}

To do this, the main point is to prove that for all $s >0$, we have $\rho((0,0,0), (0,0,s))=0$.
We will construct a sequence of horizontal curves connecting
$(0,0,0)$ to $(0,0,s)$, whose length converges to zero.
Such sequence is obtained by gluing different sequences of horizontal curves.
We fix $c >0$ and consider
$\gamma^n:[0,1] \to\vH$ defined by
$$ \gamma^n(t)= (\gamma_1^n(t), \gamma_2^n(t), \gamma_3^n(t)) = \Big (\frac{t^2}{2}c e_n , -te_nc, \frac{t^3}{6} c^2\Big),$$
where the unit vector $e_n$ is the $n$-th vector of the fixed orthonormal basis
$\{e_j:j\ge1\}$ of $\ell^2$.
By definition \eqref{def:normB}, we get  
\begin{equation}\label{eq:gamma_B}
\norm{\dot{\gamma}^n_1(t)}_\eta^2 = \frac{t^2 c^2}{n}
\quad\text{and}\quad \norm{\dot{\gamma}^n_2(t)}_\eta^2 = \frac{c^2}{n}. 
\end{equation}
From the form of $\gamma^n$, it is immediate to check that the differential constraint
$$\dot{\gamma}_3^n - \langle \gamma_1^n,\dot{\gamma}_2^n \rangle + \langle \gamma_2^n , \dot{\gamma}_1^n\rangle = 0$$
is satisfied for all $t\in [0,1]$, hence $\gamma_n$ is horizontal.  
Thus, formula \eqref{def:lg} holds and the expressions of \eqref{eq:gamma_B}
immediately prove that $\ell_g(\gamma^n)\to0$ as $n\to+\infty$.

Now we define the sequence of curves
$\alpha^n:[0,1] \to \vH$ as
$$ \alpha^n(t) =(\alpha_1^n(t),\alpha_2^n(t),\alpha_3^n(t))=\left(c\bigg(\frac{1}{2}-\frac{t^2}{2}\bigg)e_n , c(t-1) e_n , c^2 
\bigg(\frac{1}{6} + \frac{t^3}{6} - \frac{t^2}{2} + \frac{t}{2}\bigg)\right).$$
We immediately obtain
\begin{equation}\label{eq:lalpha}
\norm{\dot{\alpha}^n_1(t)}_\eta^2 = \frac{t^2 c^2}{n}\quad \text{and} \quad
\norm{\dot{\alpha}^n_2(t)}_\eta^2 = \frac{c^2}{n} 
\end{equation}
and the differential constraint
$$\dot{\alpha}_3^n - \langle \alpha^n_1,\dot{\alpha}_2^n \rangle + \langle \alpha^n_2 , \dot{\alpha}_1^n\rangle = 0$$
is satisfied for all $t\in[0,1]$.
All curves $\alpha^n$ are horizontal, hence 
combining \eqref{def:lg} and \eqref{eq:lalpha},
we conclude that $\ell_g(\alpha^n)\to 0$ as $n \to +\infty$.
We note that 
$$
\alpha^n(0)=\left(\frac{c}{2} e_n , -ce_n, \frac{c^2}{6} \right)= \gamma^n(1)
$$
for all $n \in \mathbb{N}$, hence we can consider the gluing $\alpha^n \ast \gamma^n:[0,1] \to \vH$ 
of $\alpha^n$ and $\gamma^n$, that is a piecewise smooth curve. 
Clearly $\alpha^n\ast\gamma^n$ is a horizontal curve
and for all $n \in \mathbb{N}$ we have 
$$
\alpha^n \ast \gamma^n(0)= \gamma^n(0)=(0,0,0)\quad \text{and}\quad 
\alpha^n \ast \gamma^n(1)= \alpha^n(1)=\left(0,0,\frac{c^2}{3}\right)
$$
and $\ell_g(\alpha^n \ast \gamma^n) = \ell_g(\alpha^n) + \ell_g(\gamma^n) \to 0$
as $n\to\infty$.
We have proved that 
$$
\rho\,\bigg(\big(0,0,0\big),\bigg(0,0,\frac{c^2}{3}\bigg)\bigg)=0,
$$
hence $\rho((0,0,0),(0,0,s))=0$ for all $s>0$.
By the left invariance of $\rho$, we have
\begin{equation*}
\rho((0,0,0),(0,0,-s)) = \rho((0,0,s),(0,0,0))=0,
\end{equation*}
therefore $\rho((0,0,0),(0,0,t))=0$ for every $t\in\R$.
We conclude that
\begin{equation*}
\begin{split}
\rho((p_1,p_2,s_1), (p_1,p_2,s_2)) & = \rho( (p_1,p_2,0)(0,0,s_1),(p_1,p_2,0)(0,0,s_2) ) \\
& = \rho((0,0,s_1),(0,0,s_2)) \\
& =\rho((0,0,0),(0,0,s_2-s_1))=0,
\end{split}
\end{equation*}
that proves \eqref{eq:dg0}.
As we have already observed, the inequality $d\le\rho$ is immediate, hence
for all $(p_1,p_2,s_1), (p_1,p_2,s_2)\in\vH$, we have proved that
\begin{equation}\label{eq:rhog0}
	d((p_1,p_2,s_1), (p_1,p_2,s_2))=0.
\end{equation}
This concludes the proof.
\end{proof}

\section{On the sectional curvature of a weak Riemannian Heisenberg group}\label{sect:remSectCurv}

In this section, we study the sectional curvature of $\vH$
equipped with the weak Riemannian metric $\sigma$.
From \eqref{def:sigma0} we recall the formula
\begin{equation*}
	\sigma_0((v_1,v_2,v_3),(w_1,w_2,w_3)=g_0((v_1,v_2),(w_1,w_2))+v_3w_3
\end{equation*}
for $(v_1,v_2,v_3),(w_1,w_2,w_3)\in T_0\vH$, where
\begin{equation}
	g_0((v_1,v_2),(w_1,w_2)) =\eta(v_1,w_1)+\eta(v_2,w_2) =\langle Av_1,w_1\rangle + 
	\langle Av_2,w_2\rangle 
\end{equation}
and $Ax=\sum_{k=1}^\infty x_k/k$, $x=\sum_{k=1}^\infty x_ke_k\in\ell^2$.
For every positive integer $j$, we use the notation
\[
e^1_j=(e_j,0,0),\quad e^2_j=(0,e_j,0)\quad \text{and} \quad e^3=(0,0,1),
\]
to indicate the standard orthonormal basis of 
$\vH$ seen as the Hilbert space
$\ell^2\times\ell^2\times\R$.

Since $\vH$ is connected, simply connected and nilpotent,
by \cite[Proposition~IV.2.7]{Nee06:Tow}, {\em we can identify the vectors $e^i_j$ and $e^3$ with the corresponding left invariant vector fields} of $\Lie(\vH)$. 
Such identification is used to find the sectional curvature 
of $\vH$, since it can be computed on planes of $\Lie(\vH)$.
From \eqref{eq:LieProdH}, we have the formulas
\begin{equation}
[e^1_i,e^2_j]=2\delta_{ij} e^3\quad \text{and}\quad [e^l_i,e^l_j]=0
\end{equation}
for all $i,j\ge1$ and $l=1,2$.
We consider a Lie algebra $\Lie(\G)$ of a Fréchet Lie group $\G$ 
equipped with a weak Riemannian metric $\bann$.
Following \cite[Theorem~5]{Arn66}, the point to compute the sectional curvature 
$K(X,Y)$ of a plane in $\Lie(\G)$ spanned by 
the orthonomal vectors $X,Y$ in $\Lie(\G)$ is to find the adjoint
\begin{equation}
B(X,Y)=\ad(Y)^*(X),
\end{equation}
namely for every $Z\in\Lie(\G)$ we have
\[
\ban{[Y,Z],X}=\ban{\ad(Y)(Z),X}=\ban{Z,\ad(Y)^*(X)}.
\]
For a strong Riemannian metric, \cite[Definition~[5.2.12]{AMR88}, the existence of $B(X,Y)$ is always ensured, but not for any weak Riemannian metric.

From formula (53) of \cite{Arn66}, we have
\begin{equation}\label{def:K(X,Y)Arnold}
K(X,Y)=\ban{\delta,\delta}+2\ban{\alpha,\beta}-3\ban{\alpha,\alpha}-4\ban{B_X,B_Y},
\end{equation}
where we define
\begin{align}
\delta&=\frac12\pa{B(X,Y)+B(Y,X)},\quad\beta=\frac12\pa{B(X,Y)-B(Y,X)},
\quad\alpha=\frac12[X,Y] \label{eq:delta_1} \\
B_X&=\frac12B(X,X)\quad \text{and}\quad B_Y=\frac12B(Y,Y). \label{eq:delta_2}
\end{align}
The proof of Theorem~\ref{thm:unboundedSectCurv} follows from the
application of \eqref{def:K(X,Y)Arnold} with respect to
$\sigma$ to suitable choices of planes.
We denote by $\bann_\sigma$ the scalar product induced by
the left invariant weak Riemannian metric $\sigma$ on $\Lie(\vH)$.
The associated norm on $\Lie(\vH)$ is denoted by $\|\cdot\|_\sigma$.
We assume that for $X,Y\in\Lie(\vH)$ the adjoint
\[
B_\sigma(X,Y)=\ad(Y)^*(X)
\]
with respect to $\sigma$ exists.
As a result, for $Z\in\Lie(\vH)$, by formula \eqref{eq:LieProdH}, we have
\begin{equation}\label{eq:adY*X}
\ban{\ad(Y)^*(X),Z}_\sigma=\ban{[Y,Z],X}_\sigma=
2\beta(\pi(Y),\pi(Z)) x^3,
\end{equation}
where $\pi:\vH\to\ell^2\times\ell^2$ is the canonical projection defined by 
\[
X=(\pi(X),x^3)=(\pi(X),0)+x^3e^3.
\]
We use the fixed orthonormal basis $e^1_j,e^2_j,e^3$ of $\vH$ 
with respect to the standard Hilbert product of $\ell^2\times\ell^2\times\R$,
getting
\[
\ad(Y)^*(X)=\sum_{j=1}^\infty [\ad(Y)^*(X)]^1_j e^1_j+
\sum_{j=1}^\infty [\ad(Y)^*(X)]^2_j e^2_j+[\ad(Y)^*(X)]^3e^3.
\]
Formula \eqref{eq:adY*X} yields
\begin{equation}\label{eq:adY*XjZ}
\sum_{j=1}^\infty \frac1j[\ad(Y)^*(X)]^1_j Z^1_j+
\sum_{j=1}^\infty \frac1j[\ad(Y)^*(X)]^2_j Z^2_j+[\ad(Y)^*(X)]^3Z^3=2\beta(\pi(Y),\pi(Z))x^3
\end{equation}
for arbitrary $Z=Z^3e^3+\sum_{j=1}^\infty Z^1_je^1_j+Z^2_je^2_j$.
In the case $X=\pi(X)$, formula \eqref{eq:adY*XjZ} 
shows the existence of $\ad(Y)^*(\pi(X))$ and yields
\begin{equation}\label{eq:adY^*pi(X)}
B_\sigma(\pi(X),Y)=\ad(Y)^*(\pi(X))=0.
\end{equation}
In the case $X=e^3$, again \eqref{eq:adY*XjZ} for $Z=e^1_j$ and $Z=e^2_j$ respectively, gives
\begin{equation}
[\ad(Y)^*(e^3)]^1_j=2j\beta(\pi(Y),e^1_j)\quad\text{and} \quad
[\ad(Y)^*(e^3)]^2_j=2j\beta(\pi(Y),e^2_j).
\end{equation}
For $Z=e^3$, applying \eqref{eq:adY*XjZ} we get
\begin{equation}
	[\ad(Y)^*(e^3)]^3=0.
\end{equation}
Assuming the existence of $\ad(Y)^*(e^3)$, we have shown that 
\[
B_\sigma(e^3,Y)=\ad(Y)^*(e^3)=2\sum_{j=1}^\infty j\beta(\pi(Y),e^1_j)e^1_j
+2 \sum_{j=1}^\infty j\beta(\pi(Y),e^2_j)e^2_j.
\]
Writing $Y=Y^3e^3+\sum_{j=1}^\infty (Y^1_j e^1_j+Y^2_je^2_j)$, we finally get
\begin{equation}\label{eq:Bsigmae^3Y}
B_\sigma(e^3,Y)=2\sum_{j=1}^\infty j(Y^1_je^2_j-Y^2_je^1_j).
\end{equation}
The assumption about the existence of $B_\sigma(e^3,Y)$
corresponds to the convergence of its series.
The next remark shows a choice of $Y$ for which 
the series \eqref{eq:Bsigmae^3Y} does not converge.

\begin{remark}\label{rem:undefinedSectCurv}\em
We consider the vector
\begin{equation}\label{eq:W}
W=\sum_{j=1}^\infty\frac{e^1_j}j\in\Lie(\vH),
\end{equation}
for which the series \eqref{eq:Bsigmae^3Y}
representing $B_\sigma(e^3,W)$ does not converge.
Clearly from \eqref{def:K(X,Y)Arnold} the sectional curvature 
$K_\sigma(e^3,W)$ cannot be defined.
\end{remark}

The previous remarks suggests that actually our sectional
curvature is discontinuous.

\begin{proposition}\label{prop:discont}
We consider the orthonormal elements $W_k,e^3\in\Lie(\vH)$ with $k\ge1$ and
\[
W_k=\pa{\sum_{j=1}^kj^{-3}}^{-1/2}\sum_{j=1}^k\frac{e^1_j}j\in\Lie(\vH).
\]
As the subspace $\spn\set{W_k,e^3}$ converges to $\spn\{W_\infty,e^3\}$
for $k\to\infty$, with 
\begin{equation}\label{eq:Winfty}
W_\infty=\pa{\sum_{j=1}^\infty j^{-3}}^{-1/2}\sum_{j=1}^\infty\frac{e^1_j}j\in\Lie(\vH),
\end{equation}
it follows that
\begin{equation}
K_\sigma(W_k,e^3)\to+\infty.
\end{equation}
The convergence of $\spn\set{W_k,e^3}$ to $\spn\{W_\infty,e^3\}$ is considered in the Grassmannian of the 2-dimensional planes contained in $\Lie(\vH)$.	
\end{proposition}

\begin{proof}
First of all, the pointwise convergence of $W_k$ to $W_\infty$ 
implies the convergence of $\spn\set{W_k,e^3}$ to $\spn\{W_\infty,e^3\}$.
To compute $K_\sigma(W_k,e^3)$, we first apply 
\eqref{eq:adY^*pi(X)}, getting
\begin{equation}
	B_\sigma(W_k,e^3)=\ad(e^3)^*(W_k)=0
\end{equation}
for all $k\ge1$.
From \eqref{eq:Bsigmae^3Y}, it follows that
$B_\sigma(e^3,e^1_j)=2je^2_j$, hence
\[
B_\sigma\pa{e^3,\frac{e^1_j}{j}}=2e^2_j.
\]
The bilinearity of $B_\sigma(\cdot,\cdot)$ yields
\begin{equation}
B_\sigma(e^3,W_k)=2\pa{\sum_{j=1}^kj^{-3}}^{-1/2}\sum_{j=1}^ke^2_j
\end{equation}
From \eqref{eq:delta_1}, taking $\delta=\pa{B_\sigma(W_k,e^3)+B_\sigma(e^3,W_k)}/2$, we obtain
\begin{equation}
\ban{\delta,\delta}_\sigma=
\frac14\|B_\sigma(e^3,W_k)\|_\sigma^2
=\pa{\sum_{j=1}^\infty j^{-3}}^{-1}
\Big\|\sum_{j=1}^ke^2_j\Big\|_\sigma^2=
\pa{\sum_{j=1}^\infty j^{-3}}^{-1}\sum_{j=1}^k j^{-1}
\end{equation}
From \eqref{eq:delta_1}, \eqref{eq:delta_2}, \eqref{eq:adY^*pi(X)} and \eqref{eq:Bsigmae^3Y}, we find
\begin{equation}
	\alpha=\frac12B_\sigma(W_k,W_k)=\frac12B_\sigma(e^3,e^3)=0.
\end{equation}
Finally, by formula \eqref{def:K(X,Y)Arnold}, we have proved that
\begin{equation}\label{eq:curvk}
	K_\sigma(W_k,e^3)=\ban{\delta,\delta}_\sigma=
\pa{\sum_{j=1}^\infty j^{-3}}^{-1}\sum_{j=1}^k j^{-1}\to+\infty
\end{equation}
as $k\to\infty$. This concludes the proof.
\end{proof}

\begin{proof}[Proof of Theorem~\ref{thm:unboundedSectCurv}]
Following the notation of the present section, we define
$$
a_{1j}=\sqrt{j}e^1_j\quad\text{and} \quad a_{2j}=\sqrt{j}e^2_j
$$ 
of $\Lie(\vH)$, that are orthonormal with respect to $\bann_\sigma$
and do not commute.
To apply \eqref{def:K(X,Y)Arnold} for finding 
$K_\sigma(a_{1j},a_{2j})$, we use \eqref{eq:delta_1} and \eqref{eq:delta_2}.
Due to \eqref{eq:adY^*pi(X)}, we get 
$$B_\sigma(a_{1j},a_{2j})=B_\sigma(a_{2j},a_{1j})=0.$$
As a result, we have
\begin{equation}
K_\sigma(a_{1j},a_{2j})=-3\ban{\alpha,\alpha}_\sigma=-\frac34
\|[a_{1j},a_{2j}]\|_\sigma^2=-3j^2.
\end{equation}
Now we wish to compute $K_\sigma(a_{1j},e^3)$ and $K_\sigma(a_{2j},e^3)$.
We first apply \eqref{eq:adY^*pi(X)} and \eqref{eq:Bsigmae^3Y}, getting
\begin{equation}
	B_\sigma(e^l_j,e^3)=\ad(e^3)^*(e^l_j)=0, \quad 
	B_\sigma(e^3,e^1_j)=2je^2_j\quad \text{and}\quad 
	B_\sigma(e^3,e^2_j)=-2je^1_j
\end{equation}
for all $l=1,2$ and $k\ge1$.
From \eqref{eq:delta_1}, taking $\delta=\pa{B_\sigma(a_{1j},e^3)+B_\sigma(e^3,a_{1j})}/2$, we obtain
\begin{align}
\ban{\delta,\delta}_\sigma&=\frac14\|B_\sigma(a_{1j},e^3)+B_\sigma(e^3,a_{1j})\|_\sigma^2=\frac14\|\sqrt{j}B_\sigma(e^3,e^1_j)\|_\sigma^2=
\frac j4\|2je^2_j\|_\sigma^2 \\
&=j^3\ban{e^2_j,e^2_j}_\sigma=j^3\ban{Ae^2_j,e^2_j}=j^2.
\end{align}
From \eqref{eq:delta_1}, \eqref{eq:delta_2}, \eqref{eq:adY^*pi(X)} and \eqref{eq:Bsigmae^3Y}, we find
\begin{equation}
\alpha=\frac12B_\sigma(e^1_j,e^1_j)=\frac12B_\sigma(e^3,e^3)=0.
\end{equation}
Due to the formula for the sectional curvature \eqref{def:K(X,Y)Arnold},
we have established that
\begin{equation}\label{eq:curv1}
K_\sigma(a_{1j},e^3)=\ban{\delta,\delta}_\sigma=j^2.
\end{equation}
In analogous setting $\delta=\pa{B_\sigma(a_{2j},e^3)+B_\sigma(e^3,a_{2j})}/2$,
we obtain
\begin{equation}
\ban{\delta,\delta}_\sigma=\frac14\|B_\sigma(e^3,a_{2j})\|_\sigma^2=
\frac j4\|B_\sigma(e^3,e^2_j)\|_\sigma^2=\frac j4\|2je^1_j\|_\sigma^2=
j^3\|e^1_j\|_\sigma^2=j^2.
\end{equation}
Again \eqref{eq:delta_1}, \eqref{eq:delta_2}, \eqref{eq:adY^*pi(X)} and \eqref{eq:Bsigmae^3Y} imply that
\begin{equation}
	\alpha=\frac12B_\sigma(e^2_j,e^2_j)=\frac12B_\sigma(e^3,e^3)=0.
\end{equation}
Due to \eqref{def:K(X,Y)Arnold}, we have also proved that
\begin{equation}\label{eq:curv2}
	K_\sigma(a_{2j},e^3)=\ban{\delta,\delta}_\sigma=j^2.
\end{equation}
Taking into account \eqref{eq:curv1} and \eqref{eq:curv2},
setting $b=e^3$, we have completed the proof.
\end{proof}

\begin{remark}\em
A direct verification shows that the computations
of sectional curvature, to prove Theorem~\ref{thm:unboundedSectCurv},
could be also carried out extending 
the finite dimensional formula of \cite[Lemma~1.1]{Milnor76}
for the countable structure coefficients of $\Lie(\vH)$.
These coefficients are obtained from the orthonormal
vectors $\sqrt{j}e^1_j,\sqrt{j}e^2_j,e^3$ of $\Lie(\vH)$
with respect to $\bann_\sigma$.
\end{remark}

Following the notation of the this section,
the sequence of curves whose length converges to zero in the proof of Theorem~\ref{Thm} can be written as
\begin{align*}
\gamma^j(t)&=\frac{ct^2}{2} e^1_j-cte^2_j+\frac{c^2t^3}{6}e^3\in\vH \quad \text{and}\\
\alpha^j(t)&= c\bigg(\frac{1}{2}-\frac{t^2}{2}\bigg)e^1_j+c(t-1) e^2_j + c^2 
\bigg(\frac{1}{6} + \frac{t^3}{6} - \frac{t^2}{2} + \frac{t}{2}\bigg)e^3\in\vH.
\end{align*}
It is interesting to notice that all such curves are contained in the 
span of the planes  
$$
\spn\{e^1_j,e^2_j\},\quad \spn\{e^1_j,e^3\}\quad \text{and}\quad \spn\{e^2_j,e^3\}.
$$
When these planes are seen in the Lie algebra, Theorem~\ref{thm:unboundedSectCurv}
shows that their sectional curvature blows-up, as the length of the curves
converges to zero.

\bibliography{References}{}

\begin{thebibliography}{1}

\bibitem{AMR88}
Ralph~H. Abraham, Jerrold~E. Marsden, and Tudor~S. Ratiu.
\newblock {\em Manifolds, tensor analysis, and applications}, volume~75 of {\em
  Applied Mathematical Sciences}.
\newblock Springer-Verlag, New York, second edition, 1988.

\bibitem{Arn66}
Vladimir~I. Arnold.
\newblock Sur la g\'{e}om\'{e}trie diff\'{e}rentielle des groupes de {L}ie de
  dimension infinie et ses applications \`a l'hydrodynamique des fluides
  parfaits.
\newblock {\em Ann. Inst. Fourier (Grenoble)}, 16(fasc., fasc. 1):319--361,
  1966.

\bibitem{BBM14}
Martin Bauer, Martins Bruveris, and Peter~W. Michor.
\newblock Overview of the geometries of shape spaces and diffeomorphism groups.
\newblock {\em J. Math. Imaging Vision}, 50(1-2):60--97, 2014.

\bibitem{EP1993}
Yakov Eliashberg and Leonid Polterovich.
\newblock Bi-invariant metrics on the group of {H}amiltonian diffeomorphisms.
\newblock {\em Internat. J. Math.}, 4(5):727--738, 1993.

\bibitem{MT20}
Valentino Magnani and Daniele Tiberio.
\newblock A remark on vanishing geodesic distances in infinite dimensions.
\newblock {\em Proc. Amer. Math. Soc.}, 148(8):3653--3656, 2020.

\bibitem{MM05}
Peter~W. Michor and David Mumford.
\newblock Vanishing geodesic distance on spaces of submanifolds and
  diffeomorphisms.
\newblock {\em Doc. Math.}, 10:217--245, 2005.

\bibitem{MM06}
Peter~W. Michor and David Mumford.
\newblock Riemannian geometries on spaces of plane curves.
\newblock {\em J. Eur. Math. Soc. (JEMS)}, 8(1):1--48, 2006.

\bibitem{Milnor76}
John Milnor.
\newblock Curvatures of left invariant metrics on {L}ie groups.
\newblock {\em Advances in Math.}, 21(3):293--329, 1976.

\bibitem{Nee06:Tow}
Karl-Hermann Neeb.
\newblock Towards a {L}ie theory of locally convex groups.
\newblock {\em Jpn. J. Math.}, 1(2):291--468, 2006.

\end{thebibliography}
\bibliographystyle{plain}

\end{document}